\documentclass{article}
\usepackage{amsmath,amssymb,amsthm,upref}
\usepackage[ansinew]{inputenc}
\usepackage{dsfont}
\usepackage{mathrsfs}
\usepackage[all]{xy}
\usepackage{html}

\newcommand{\Dl}{\ensuremath{\Delta} }
\newcommand{\simp}{\Delta^\comp}

\newcommand{\mc}[1]{\mathcal{#1}}
\newcommand{\mrm}[1]{\mathrm{#1}}
\newcommand{\mbf}[1]{\mathbf{#1}}

\newcommand{\comp}{{\scriptscriptstyle \ensuremath{\circ}}}

\newtheorem{thm}{\bf T{\footnotesize HEOREM}}[section]
\newtheorem{lema}[thm]{\bf L{\footnotesize EMMA}}
\newtheorem{prop}[thm]{\bf P{\footnotesize ROPOSITION}}
\newtheorem{cor}[thm]{\bf C{\footnotesize OROLLARY}}
\theoremstyle{definition}
\newtheorem{defi}[thm]{\bf D{\footnotesize EFINITI{O}N}}
\newtheorem{obs}[thm]{\bf R{\footnotesize EMARK}}
\newtheorem{ej}[thm]{\bf E{\footnotesize XAMPLE}}
\newtheorem{ejs}[thm]{\bf E{\footnotesize XAMPLES}}

\begin{document}

\title{Triangulated structures induced by simplicial descent categories}
\author{Beatriz Rodr\'iguez Gonz\'alez\footnote{Work on this paper was partially supported by the research projects `ERC Starting Grant TGASS', \href{http://www.grupo.us.es/gfqm218}{`Geometr\'{\i}a Algebraica, Sistemas Diferenciales y Singularidades' FQM-218}, by MTM2007-66929 and by FEDER}\mbox{}\footnote{email: rgbea@imaff.cfmac.csic.es}\\
\begin{small}Instituto de Ciencias Matem{\'a}ticas (ICMAT)\end{small}\\
\begin{small}CSIC-UAM-UC3M-UCM\end{small}}
\date{}

\maketitle

\begin{abstract} The present paper is devoted to study the homotopy category associated with a simplicial descent category $(\mc{D},\mbf{s},\mrm{E})$ \cite{R1}. We prove that the class $\mrm{E}$ of equivalences has a calculus of left fractions over a quotient category of $\mc{D}$ modulo homotopy. We study the fiber/cofiber sequences induced by a (co)simplicial descent structure. Examples of such fiber/cofiber sequences are deduced for (commutative) differential graded algebras, simplicial sets or topological spaces. We prove that the homotopy category of a stable simplicial descent category is triangulated. In addition, these triangulated structures may be extended to the homotopy categories of diagram categories of $\mc{D}$. As a corollary, we obtain the triangulated structures on: (filtered) derived categories of abelian categories, the derived category of DG-modules over a DG-category, the stable derived category of fibrant spectra and the localized category of mixed Hodge complexes.
\end{abstract}

\section*{Introduction}

 Given a category $\mc{D}$ and a class of equivalences $\mrm{E}$, there are multiple approaches to induce homological or homotopical structure on the associated homotopy category $Ho\mc{D}=\mc{D}[\mrm{E}^{-1}]$. A common feature among the induced properties uses to be the existence of fiber or cofiber sequences in $Ho\mc{D}$.

 Classical examples are the properties of the fiber/cofiber sequences of topological spaces or simplicial sets (see \cite[chapter V]{GZ} to see a treatment based on 2-categories and groupoids). In the additive/abelian context, fiber and cofiber sequences agree and are called `distinguished triangles'. Verdier summarized their properties in the notion of \textit{triangulated category}, following Grothendieck ideas.

 In \cite{Q}, Quillen presented the notion of model categories as a general framework to do homological algebra. Among other applications, a Quillen model structure on a pointed category $\mc{M}$ produces fiber and cofiber sequences in $\mc{M}$. In the stable case, the homotopy category $Ho\mc{M}$ of a Quillen model category is triagulated in the sense of Verdier.

 Later, Grothendieck developed the theory of (triangulated) \textit{derivators} in an unpublished manuscript, that is being edited by Maltsiniotis \cite{Malt}. One of the relevant properties of a triagulated derivator is that it produces compatible triagulated structures on all $Ho (I,\mc{D})$, where $(I,\mc{D})$ are categories of diagrams in $\mc{D}$.

 The present paper is devoted to study the cofiber sequences in a category $\mc{D}$ that are induced by the natural cofiber sequences existing in $\simp\mc{D}$. Just assuming that $\mc{D}$ has finite coproducts and final object, then there are cofiber sequences in $\simp\mc{D}$. They are defined through the classical cone and suspension functors (that are constructed as in $\simp Set$). We carry them to $\mc{D}$ using a sort of `simple' (or total) functor $\mbf{s}:\simp\mc{D}\rightarrow\mc{D}$. More concretely, we assume that $(\mc{D},\mrm{E},\mbf{s})$ is a \textit{\textit{simplicial descent category}} \cite{R1}. Some of the most remarkable axioms they must satisfy are the following\\
 \textsc{Coproducts}: $\mc{D}$ has finite coproducts. The simple functor commutes with coproducts up to equivalence, and $\mrm{E}\sqcup\mrm{E}\subseteq\mrm{E}$.\\
 \textsc{Exactness}: The simple of a degree-wise equivalence is an equivalence.\\
 \textsc{Acyclicity}: The simple of the simplicial cone of a map $f$ is acyclic if and only if $f\in\mrm{E}$.\\
 \textsc{Eilenberg-Zilber}: The iterated simple of a bisimplicial object is equivalent to the simple of its diagonal.

 Here we prove that a simplicial descent category $\mc{D}$ inherits from $\simp\mc{D}$ the homotopic properties one would expect. First of all, we consider the homotopy relation on $\mc{D}$ defined through the induced cylinder. The second section of the present paper contains the

 \noindent {\bf T{\footnotesize HEOREM} \ref{MorfHoD} }
 \textit{If $(\mc{D},\mrm{E},\mbf{s})$ is a simplicial descent category, then the class $\mrm{E}$ has a calculus of left fractions in the quotient category of $\mc{D}$ modulo homotopy.}

 In the third section we study the cofiber sequences in $Ho\mc{D}$ induced by those of $\simp\mc{D}$. Although $Ho\mc{D}$ may be non-additive (even non-pointed), we define an `homotopical' minus sign $m_B:\Sigma B\rightarrow\Sigma B$. On the other hand, the suspension functor $\Sigma$ does not need to be an equivalence of categories. However, all `non-stable' axioms of triangulated category are satisfied by our cofiber sequences. Instead of the classical TR 2, it holds that whenever $A\stackrel{f}{\rightarrow} B\rightarrow C\rightarrow \Sigma A$ is a cofiber sequence, then so is $B\rightarrow C\rightarrow \Sigma A\stackrel{m\Sigma f}{\rightarrow} \Sigma B$.
 Dually, if our category $\mc{D}$ is a cosimplicial descent category, we construct \textit{fiber sequences} with the dual properties. In particular, the result applies to non-pointed examples as differential graded algebras and commutative differential graded algebras over a field of zero characteristic.

 In case our simplicial descent category is additive, the previous properties mean that $Ho\mc{D}$ is a suspended category \cite{KV}. As an application, we get suspended category structures on the filtered derived category of uniformly-bounded-below cochain complexes, filtered by a birregular filtration. This allows us to deduce the usual triangulated category structure on the derived category of bounded-below birregularly-filtered complexes. Using the same technique, we also obtain a triagulated structure on the category of bounded-below birregularly-filtered complexes localized with respect to the $E_2$-isomorphisms. In addition, Deligne's `decalage' functor is a triangulated functor between these two localized categories. We restrict our attention to birregular filtrations because we are interested in mixed Hodge complexes.

 In section 4 we treat the pointed case, that is, a simplicial descent category $\mc{D}$ such that $Ho\mc{D}$ is pointed. We prove that the suspension $\Sigma A$ of $A$ is a cogroup object in $Ho\mc{D}$, and that it coacts on the cone of a map $f:A\rightarrow B$.

 Finally, in section 5 we study the stable case, that is, when $\Sigma:Ho\mc{D}\rightarrow Ho\mc{D}$ is an equivalence of categories. In this case, $Ho\mc{D}$ is pointed. Therefore, the group structures on $Hom_{Ho\mc{D}}(\Sigma A,-)$ endow $Ho\mc{D}$ with an additive category structure. Hence, we have the

 \noindent {\bf T{\footnotesize HEOREM} \ref{StablTriang}} \textit{The homotopy category $Ho\mc{D}$ of a stable simplicial descent category is a Verdier triangulated category.}

 Moreover, the resulting triangulated structures extend to diagram categories as happens with derivators.

 \noindent {\bf C{\footnotesize OROLLARY} \ref{cORStablTriang}} \textit{Let $\mc{D}$ be a stable simplicial descent category such that $\Sigma:\mc{D}\rightarrow \mc{D}$ is an equivalence of categories. Then, if $I$ is a small category, the category of diagrams $(I,\mc{D})$ is a stable simplicial descent category. In particular, $Ho(I,\mc{D})$ is a Verdier's triangulated category, and any functor $f:I\rightarrow J$ induces a triangulated functor $f^\ast: Ho(J,\mc{D})\rightarrow Ho(I,\mc{D})$.}

 Using the simplicial and cosimplicial descent categories exhibited in \cite{R1}, the previous theorem produces the well-known triangulated structures on

 \begin{itemize}
 \item The homotopy category of an additive category, and the derived category of an abelian category.
 \item The derived category of DG-modules over a DG-category.
 \item The filtered derived category of an abelian category, and the category of filtered cochain complexes localized with respect to $E_2$-isomorphisms.
 \item The derived category of mixed Hodge complexes.
 \item The stable derived category of fibrant spectra.
 \end{itemize}

I wish to thank L. Narv\'aez Macarro and V. Navarro Aznar for their helpful advice, expert guidance and dedication.

\section{Simplicial descent categories}

Let $\simp\mc{D}$ (resp. $\simp\simp\mc{D}$) be the category of simplicial (resp. bisimplicial) objects in a fixed category $\mc{D}$ (see, for instance \cite{GZ}). The constant simplicial object defined by an object $A$ of $\mc{D}$ will be denoted by $A\times \Dl$. Also, we denote by $\Upsilon:\simp\mc{D}\rightarrow\simp\mc{D}$ the `order inverse' functor. It is defined as $(\Upsilon X)_n=X_n$, $\Upsilon(d^n_i)=d^n_{n-i}$ and $s^n_j=s^n_{n-j}$.

\begin{defi}\cite[2.2]{R1}
A \textit{simplicial descent category} is the data $(\mc{D},\mrm{E},\mbf{s},\mu,\lambda)$ satisfying
the following axioms\\%
$\mathbf{(S\; 1)}$ $\mc{D}$ is a category with finite
coproducts, initial object $0$ and final object $\ast$.\\[0.1cm]
$\mathbf{(S\; 2)}$ $\mrm{E}$ is a saturated class of morphisms
in $\mc{D}$, called \textit{equivalences}, stable by coproducts (that is
$\mrm{E}\sqcup\mrm{E}\subseteq \mrm{E}$). An object $A$ is \textit{acyclic} if $A\rightarrow\ast$ is in $\mrm{E}$.\\[0.1cm]
$\mathbf{(S\; 3)}$ $\mbf{s}:\simp\mc{D}\rightarrow \mc{D}$ is a functor, called the \textit{simple functor}, which commutes with coproducts up to equivalence.\\[0.1cm]
$\mathbf{(S\; 4)}$ $\mu:\mbf{s}\mrm{D}\rightarrow\mbf{s}\mathbf{s}$ is an isomorphism
of $Fun(\simp\simp\mc{D},\mc{D})[\mrm{E}^{-1}]$, the category of functors from $\simp\simp\mc{D}$ to $\mc{D}$ localized by the object-wise equivalences. If $Z\in\simp\simp\mc{D}$, then $\mbf{s}\mrm{D}Z$ is the simple of the diagonal of $Z$. On the other hand
$\mbf{s}\mbf{s}Z=\mbf{s}(n\rightarrow\mbf{s}(m\rightarrow
Z_{n,m}))$ is the iterated simple of $Z$.\\[0.1cm]
$\mathbf{(S\; 5)}$ $\lambda:\mbf{s}(-\times\Dl)\rightarrow
Id_{\mc{D}}$ is a natural transformation such that $\lambda_X\in \mrm{E}$ for all $X\in\mc{D}$. In addition, there exists $\rho:Id_{\mc{D}}\rightarrow \mbf{s}(-\times\Dl)$ with $\rho\lambda=Id$.
We will write $\mrm{R}:=\mbf{s}(-\times\Dl):\mc{D}\rightarrow\mc{D}$.\\[0.1cm]
$\mathbf{(S\; 6)}$ If $f:X\rightarrow Y$ is a morphism in
$\simp\mc{D}$ with $f_n \in \mrm{E}$ for all $n$,
then $\mbf{s}(f)\in\mrm{E}$.\\[0.1cm]
$\mathbf{(S\; 7)}$
If $f:A\rightarrow B$ is a morphism in $\mc{D}$, then
$f\in\mrm{E}$ if and only if $\mbf{s}C(f\times\Dl)$, the simple of its simplicial
cone, is acyclic.\\[0.1cm]
$\mathbf{(S\; 8)}$ It holds that $\mbf {s}\Upsilon f\in\mrm{E}$
if (and only if) $\mbf{s}f\in\mrm{E}$, where
$\Upsilon:\simp\mc{D}\rightarrow\simp\mc{D}$ is the inverse order functor.\\[0.1cm]
The transformations $\mu$ and $\lambda$ of $\mathbf{(S\; 4)}$ and $\mathbf{(S\; 5)}$ should be compatible in the following sense. If $X\in\simp\mc{D}$, then the
compositions below are equal to the identity in $Fun(\mc{D},\mc{D})[\mrm{E}^{-1}]$.
\begin{equation}\label{compatibLambdaMuEquac}\xymatrix@M=4pt@H=4pt@R=7pt@C=25pt{
 \mbf{s}X \ar[r]^-{\mu_{\Dl\times X}} &  \mrm{R}\mbf{s}X\ar[r]^-{\lambda_{\mbf{s}X}} & \mbf{s}X &
 \mbf{s}X \ar[r]^-{\mu_{X\times \Dl}} &  \mbf{s}(\mrm{R}X)\ar[r]^-{\mbf{s}(\lambda_{X})} &
 \mbf{s}X}\end{equation}
A \textit{cosimplicial descent category} is a category $\mc{D}$ such that its opposite category $\mc{D}^{op}$ is a simplicial descent category.
\end{defi}

\begin{defi}
A \textit{descent functor} between simplicial descent categories $(\mc{D},\mbf{s},\mrm{E},\mu,\lambda)$ and
$(\mc{D'},\mrm{E'},\mbf{s'},\mu',\lambda')$ is a functor
$\psi:\mc{D}\rightarrow \mc{D}'$ such that\\
\textbf{I.-} $\psi(\mrm{E})\subseteq \mrm{E}'$.\\
\textbf{II.-} The canonical map $\psi(A)\sqcup\psi(B)\rightarrow\psi(A\sqcup B)$ is
in $\mrm{E}'$ for each $A,B\in\mc{D}$.\\
\textbf{III.-} There exists an isomorphism $\Theta:\psi\mbf{s}\rightarrow\mbf{s}'\psi$ of $Fun(\simp\mc{D},\mc{D}')[\mrm{E}'^{-1}]$, compatible with $\lambda$, $\lambda'$
and with $\mu$, $\mu'$ (see \cite{R1} for the explicit compatibility conditions).\\
We will also assume that $\psi$ preserves final objects, that is $\psi(\ast)\rightarrow\ast'$ is in $\mrm{E}$.
\end{defi}


Next we study the cylinder objects in $\simp\mc{D}$, and those induced by them in $\mc{D}$. Consider the natural action
$\simp Set\times \simp\mc{D}\rightarrow \simp\mc{D}$, given by $(K\boxtimes X)_n=\coprod_{K_n}X_n$.

If $f:X\rightarrow Y$ and $g:X\rightarrow Z$ are maps in $\simp\mc{D}$,
then $Cyl(f,g)$ denotes the pushout of $f\sqcup g:X\sqcup X\rightarrow Y\sqcup Z$ along $d^0\sqcup d^1:X\sqcup X\rightarrow X\boxtimes \Dl[1]$, so $Cyl(f,g)_n=Y_n\sqcup \coprod^n X_n\sqcup Z_n$. If $g:X\rightarrow \ast$ is the trivial map, then $Cyl(f,g)$ is the \textit{cone} of $f$, denoted by $C(f)$.

We next recall a basic property of $Cyl$ that will be often used in section \ref{Cofsequenc}.

Consider the following commutative diagram in $\simp\mc{D}$
$$\xymatrix@M=4pt@H=4pt@R=12pt@C=12pt{Z'  & X'\ar[l]_{g'}\ar[r]^{f'}   & Y' \\
                        Z  \ar[u]^{\alpha} \ar[d]_{\alpha'}& X  \ar[l]_{g}\ar[r]^{f} \ar[u]^{\beta} \ar[d]_{\beta'} & Y \ar[u]^{\gamma} \ar[d]_{\gamma'} \\
                        Z'' & X''\ar[l]_{g''}\ar[r]^{f''} & Y''}$$
Applying $Cyl$ by rows and columns we obtain
$$\xymatrix@M=4pt@H=4pt@R=7pt@C=12pt{
 Cyl(f',g')& Cyl(f,g)\ar[l]_-{\delta}\ar[r]^-{\delta'}& Cyl(f'',g'') &  Cyl(\alpha',\alpha)&
 Cyl(\beta',\beta)\ar[l]_-{\widehat{g}}\ar[r]^-{\widehat{f}} & Cyl(\gamma',\gamma)}$$
Let $\psi:Cyl(\gamma',\gamma)\rightarrow Cyl(\delta',\delta)$ and
$\psi':Cyl(f'',g'')\rightarrow Cyl(\widehat{f},\widehat{g})$ be
the respective images under $Cyl$ of:
$$\xymatrix@M=4pt@H=4pt@R=15pt@C=12pt{Y'\ar[d]_{I}& Y\ar[l]_-{\gamma}\ar[r]^-{\gamma'}\ar[d]_{I} & Y''\ar[d]_{I}\ar@{}[rd]|{;} & Z''\ar[d]_{I}& X''\ar[l]_-{g''}\ar[r]^-{f''}\ar[d]_{I} & Y''\ar[d]_{I}\\
                         Cyl(f',g')& Cyl(f,g)\ar[l]_-{\delta}\ar[r]^-{\delta'}& Cyl(f'',g'')                                          & Cyl(\alpha',\alpha)& Cyl(\beta',\beta)\ar[l]_-{\widehat{g}}\ar[r]^-{\widehat{f}}& Cyl(\gamma',\gamma)}$$
where each $I$ means the corresponding canonical map.

\begin{lema}\textrm{\cite[0.1.2]{R1}}\label{CilindroIterado}
There exists a natural isomorphism of simplicial objects
$\Theta:Cyl(\delta',\delta){\rightarrow}
Cyl(\widehat{f},\widehat{g})$, such that $\Theta I=\psi'$ and
$\Theta \psi=I$.
\end{lema}

From now on, $(\mc{D},\mrm{E},\mbf{s},\mu,\lambda)$ denotes a simplicial descent category.
Given an object $A$ in $\mc{D}$, we can consider the associated constant simplicial object $A\times\Dl\in\simp\mc{D}$. Since $\mc{D}$ has finite coproducts, the classical cylinder object $A\boxtimes\Dl[1]$ gives rise to the diagram
$$\xymatrix@R=15pt@C=20pt{A\times\Dl \ar@<0.5ex>[r]^-{d^0}
\ar@<-0.5ex>[r]_-{d^1}& \Dl[1]\boxtimes A \ar[r]^-{s^0} & A\times\Dl }\mbox{ with } s^0 d^0=s^0 d^1 =Id_A$$
Set $cyl(A)=\mbf{s}(\Dl[1]\boxtimes A)$. Applying the simple functor $\mbf{s}$ to the above diagram, and composing $\mbf{s}d_0$, $\mbf{s}d_1$ with $\rho_A$ and ${s_0}$ with $\lambda_A$, we get the following diagram in $\mc{D}$
\begin{equation}\label{cylX}\xymatrix@R=15pt@C=20pt{A\ar@<0.5ex>[r]^-{i_A}
\ar@<-0.5ex>[r]_-{j_A}& cyl{A}\ar[r]^{\sigma_A} & A }\mbox{ with } \sigma_A i_A=\sigma_A j_A =Id\end{equation}

\begin{defi}
If $C\stackrel{g}{\leftarrow}A\stackrel{f}{\rightarrow} B$ are maps in $\mc{D}$, we will write $cyl(f,g)$ for $\mbf{s}Cyl(f\times\Dl,g\times\Dl)$. As before, the simplicial maps $J_C:C\times\Dl\rightarrow Cyl(f,g)$ and $I_B:B\times\Dl\rightarrow Cyl(f,g)$ provide $j_C:C\rightarrow cyl(f,g)$ and  $i_B:B\rightarrow cyl(f,g)$, where $i_C=\mbf{s}J_C\rho_C$ and $i_B=\mbf{s}I_B \rho_B$. In this way $cyl$ gives rise to a functor that maps the pair $(f,g)$ to the data $(cyl(f,g),i_B,j_C)$.
\end{defi}

Note that if $p:B\rightarrow T$, $q:C\rightarrow T$ are such that $pf=qg$, then there is a natural $r:cyl(f,g)\rightarrow T$ with $r j_C=q$ and $r i_B=p$. Indeed, just take $r=\sigma_T cyl(q,Id_A,p)$, where $cyl(q,Id_A,p):cyl(f,g)\rightarrow cyl(T)$ is the map induced by $q$, $Id_A$ and $p$.

Next we recall some properties concerning functor $cyl$.

\begin{lema}\label{C3}\textrm{\cite[0.2.7]{R1}} Consider a commutative diagram in $\simp\mc{D}$
$$\xymatrix@M=4pt@H=4pt@C=15pt@R=15pt{ Z \ar[d]_-{\alpha} & X \ar[r]^-{f} \ar[l]_-{g}\ar[d]_{\beta} &  Y \ar[d]^{\gamma}\\
                                       Z'                 & X' \ar[r]^-{f'} \ar[l]_-{g'}             &  Y'   ,}$$
such that $\mbf{s}\alpha$, $\mbf{s}\beta$ and $\mbf{s}\gamma$ are equivalences. Then the induced morphism $\mbf{s}Cyl(f,g)\rightarrow
\mbf{s}Cyl(f',g')$ is also in $\mrm{E}$. In particular, if $\alpha$, $\beta$ and $\gamma$ are constant maps which are equivalences, then the induced map $cyl(f,g)\rightarrow cyl(f',g')$ is so.
\end{lema}

\begin{lema}\label{sCyls}\textrm{\cite[0.2.8]{R1}} If $f:X\rightarrow Y$ and $g:X\rightarrow Z$ are maps
in $\simp\mc{D}$, then $\mbf{s}Cyl(f,g)$ is naturally isomorphic
to $T=\mbf{s}C((\mbf{s}f)\times\Dl,(\mbf{s}g)\times\Dl)$ in
$Ho\mc{D}$. This isomorphism commutes with the respective
canonical maps from $\mrm{R}\mbf{s}Y$, $\mrm{R}\mbf{s}Z$, $\mbf{s}Y$,
and $\mbf{s}Z$ to $T$ and $\mbf{s}Cyl(f,g)$.
\end{lema}

\begin{lema}\cite[0.2.9]{R1}\label{C2}
Given maps $C\stackrel{g}{\leftarrow}A\stackrel{f}{\rightarrow} B$ in $\mc{D}$, then\\
\textbf{a)} $f\in\mrm{E}$ implies that $j:C\rightarrow cyl(f,g)$ is an equivalence too.\\
\textbf{b)} $g\in\mrm{E}$ implies that $i:B\rightarrow cyl(f,g)$ is an equivalence too.
\end{lema}

In particular, the maps $i_A,j_A$ and $\sigma_A$ in (\ref{cylX}) are equivalences.


\section{The homotopy category}
In this section we describe the morphisms in the localized category $Ho\mc{D}=\mc{D}[\mrm{E}^{-1}]$. We prove that $\mrm{E}$ has a calculus of left fractions over a quotient of $\mc{D}$ modulo homotopy. Then, any map in $Ho\mc{D}=\mc{D}[\mrm{E}^{-1}]$ is represented by a `roof' $A\rightarrow T\stackrel{\thicksim}{\leftarrow} B$.

To this end, we would like to consider homotopies between maps
$f,g:X\rightarrow B$ in $\mc{D}$ defined through the cylinder $cyl(X)$. The
problem is that the relation obtained in this way is not transitive in general. To solve this problem, we will consider the associated equivalence relation, or equivalently, we will enlarge the set of cylinders than can be used to define homotopies.

Assume that $\xymatrix@C=15pt@R=15pt{A \ar@<0.5ex>[r]^-{i} \ar@<-0.5ex>[r]_-{j}& \widetilde{A} \ar[r]^-{\nu}& A}$ and
$\xymatrix@C=15pt@R=15pt{A \ar@<0.5ex>[r]^-{p}
\ar@<-0.5ex>[r]_-{q}& \widehat{A} \ar[r]^-{\varrho}& A}$ are
diagrams such that $\nu i=\nu j=\varrho p=\varrho q=Id_A$.
We can `\textit{glue}' them obtaining a new diagram $\xymatrix@C=15pt@R=15pt{A
\ar@<0.5ex>[r]^-{s} \ar@<-0.5ex>[r]_-{t}& \overline{A}
\ar[r]^-{\eta}& A}$ with $\eta s=\eta t=Id$, in the following way. Applying $cyl$ to
$\xymatrix@C=15pt@R=15pt{\widetilde{A} & \ar[r]^-{p}  A \ar[l]_-{j} &
\widehat{A}}$ we get
$\xymatrix@C=15pt@R=15pt{\widetilde{A} \ar[r]^-{\alpha} & \overline{A} & \ar[l]_{\beta} \widehat{A}}$. This data can be summarized in the diagram
$$
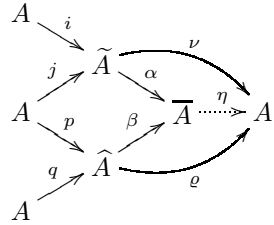
\xymatrix@C=17pt@R=4pt{
 A \ar[rd]^{i} & & & \\
 & \widetilde{A} \ar[rd]^-{\alpha} \ar@/^1pc/[rrd]^-{\nu} & \\
 A\ar[ru]^-{j}\ar[rd]^-{p} & & \overline{A}\ar@{.>}[r]^-{\eta} & A\\
 & \widehat{A} \ar[ru]^-{\beta}\ar@/_1pc/[rru]_-{\varrho} & & \\
 A \ar[ru]^-{q} & & &}$$
Define $s=\alpha i$ and $t=\beta q$. Let us define
$\eta$. Since $\nu j=\varrho p= Id_A$, then
$cyl(\varrho,Id_A,\nu):cyl(p,j)=\overline{A}\rightarrow cyl(A)$. Set
$\alpha=\nu_A cyl(\varrho,Id_A,\nu)$, where
$\nu_A:cyl(A)\rightarrow A$ is the canonical map. As
$\eta\alpha=\nu$ and $\eta \beta=\varrho$. Then
$s,t,\eta$ are such that $\eta s=\eta t=Id$.

\begin{defi} The set of
\textit{admissible cylinders} of an object $A$ of $\mc{D}$ is defined inductively as follows\\
I) The diagram $\xymatrix@C=15pt@R=15pt{A \ar@<0.5ex>[r]^-{Id} \ar@<-0.5ex>[r]_-{Id}& A \ar[r]^{Id}& A }$ is an admissible cylinder.\\
II) If $\xymatrix@C=15pt@R=15pt{A \ar@<0.5ex>[r]^-{i} \ar@<-0.5ex>[r]_-{j}& \widetilde{A} \ar[r]^-{\sigma}&
 A}$ is an admissible cylinder, so is $\xymatrix@C=15pt@R=15pt{A \ar@<0.5ex>[r]^-{j} \ar@<-0.5ex>[r]_-{i}& \widetilde{A} \ar[r]^-{\sigma}& A}$.\\
III) The result of
gluing two admissible cylinders is again an admissible cylinder.
\end{defi}

\begin{obs}\label{InclusinesenE}
It can be proved inductively that
given an admissible cylinder $\xymatrix@C=15pt@R=15pt{A
\ar@<0.5ex>[r]^-{i} \ar@<-0.5ex>[r]_-{j}& \widetilde{A}
\ar[r]^-{\sigma}& A}$ then $\sigma i=\sigma j=Id$. By lemma \ref{C2}, we deduce that $i$, $j$ and $\sigma$ are in
$\mrm{E}$.
\end{obs}

\begin{defi} Two maps $f,g:A\rightarrow B$ are
\textit{homotopic} if there exists an admissible cylinder
$\xymatrix@C=15pt@R=15pt{A \ar@<0.5ex>[r]^-{i}
\ar@<-0.5ex>[r]_-{j}& \widetilde{A} \ar[r]^-{\sigma}& A}$ and
$H:\widetilde{A}\rightarrow B$ such that $H i=f$ and $H j=g$. $H$ is called a \textit{homotopy} from $f$ to $g$.
\end{defi}

\begin{ej}\label{ejemplHomot} Given maps $\xymatrix@C=15pt@R=15pt{C  & A\ar[l]_-g\ar[r]^f &
B}$ in $\mc{D}$, then  $j_C f,i_B g:A\rightarrow cyl(f,g)$ are homotopic maps through $H=cyl(g,Id,f):cyl(A)\rightarrow cyl(f,g)$.
\end{ej}

\begin{lema}\label{relHomotopia}
The homotopy relation on $\mrm{Hom}_{\mc{D}}(A,B)$ is an equivalence relation.
It is compatible with right and left composition of morphisms in $\mc{D}$.
It is also compatible with $\mrm{E}$, that is, if $f$ and $g$ are homotopic maps then $f=g$ in $Ho\mc{D}$. In particular, $f\in\mrm{E}$ if and only if $g\in\mrm{E}$.
\end{lema}

\begin{proof} First assertion follows directly from definitions.
Let $H:\widetilde{A}\rightarrow B$ be a homotopy from $f$ to $g$, where  $i,j:A\rightarrow  \widetilde{A}$ are part of an admissible cylinder $(i,j,\sigma)$ of $A$. If $t:B\rightarrow C$ then $t f$ is clearly
homotopic to $t g$ through the homotopy $t H$. If
$s:S\rightarrow A$,
it can be proved by induction that there exists an admissible
cylinder $\xymatrix@C=15pt@R=15pt{S \ar@<0.5ex>[r]^-{p}
\ar@<-0.5ex>[r]_-{q}& \widetilde{S}}$ and a map $\widetilde{s}:\widetilde{S}\rightarrow \widetilde{A}$ such that $i s=\widetilde{s} p$ and
$j s=\widetilde{s} q$. Hence $H \widetilde{s}$ is a homotopy from $f s$ to $g s$.
To see the last statement, note that $i=j=\sigma^{-1}$ in $Ho\mc{D}$, by remark \ref{InclusinesenE}.
\end{proof}

Denote by $K\mc{D}$ the category with same objects as
$\mc{D}$, and whose morphisms are the morphisms of $\mc{D}$ modulo
homotopy. We will write $[f]$ for the homotopy class in $K\mc{D}$
of a morphism $f$ of $\mc{D}$.

\begin{thm}\label{MorfHoD}
$Ho\mc{D}=\mc{D}[\mrm{E}^{-1}]$
is isomorphic to the localized category of $K\mc{D}$ with respect to the
class $\mrm{E}$ modulo homotopy. In addition the class $\mrm{E}$ admits a
calculus of left fractions in $K\mc{D}$.
\end{thm}

\begin{proof}
The first statement follows trivially from the universal properties of $Ho\mc{D}$, $K\mc{D}$ and $K\mc{D}[\mrm{E}^{-1}]$.\\
Let us prove that $\mrm{E}$ has a calculus of left fractions in $K\mc{D}$. Given maps $\xymatrix@C=15pt@R=15pt{C&A\ar[l]_-{[e]} \ar[r]^-{[f]}&
B}$ in $K\mc{D}$ with $e\in\mrm{E}$, then the square
$$\xymatrix@C=15pt@R=15pt{A \ar[r]^-{f}\ar[d]_-e & B\ar[d]_-i \\
            C  \ar[r]^-j          & cyl(f,e)}$$
is such that $i\in\mrm{E}$ by lemma \ref{C2}, and it commutes up to homotopy by example \ref{ejemplHomot}.
On the other hand, let $\xymatrix@C=15pt@R=15pt{A' \ar[r]^-s & A
\ar@<0.5ex>[r]^-{f} \ar@<-0.5ex>[r]_-{g}& B}$ be morphisms in $\mc{D}$ such
that $s\in\mrm{E}$ and $f s$ is homotopic to $g s$ through the admissible cylinder $\xymatrix@C=15pt@R=15pt{A' \ar@<0.5ex>[r]^-{i}
\ar@<-0.5ex>[r]_-{j}& \widetilde{A'} \ar[r]^-{\nu}& A'}$ and the
homotopy $H:\widetilde{A'}\rightarrow B$. Summing all up, we get the commutative diagram
$$\xymatrix@C=15pt@R=15pt{
 A \ar[d]_{Id} & A' \ar[l]_-{s} \ar[r]^-{fs}\ar[d]_-{i} & B\ar[d]^-{Id}\\
 A & \widetilde{A'} \ar[r]^-{H}\ar[l]_-{s\nu} & B \\
 A \ar[u]^-{Id} & A'\ar[l]_-{s}\ar[u]^-{j} \ar[r]^-{gs} & B \ar[u]_-{Id}\ .}$$
Applying functor $cyl$ by rows we obtain
\begin{equation}\label{diagrAux0}\xymatrix@C=15pt@R=15pt{
 A \ar[d]_{Id} \ar[r]^-{p} & cyl(s,fs) \ar[d]_-{\alpha} & B\ar[d]^-{Id}\ar[l]_-{q}\\
 A \ar[r] & cyl(s\nu,H)   & B\ar[l]_-w \\
 A \ar[u]^-{Id} \ar[r]^-{k} & cyl(s,gs) \ar[u]^-{\beta} & B\ar[u]_-{Id}\ar[l]_-{l}}
 \end{equation}
where $q$, $w$ and $l$ are equivalences by lemma \ref{C2}. Now there are equivalences $\gamma:cyl(s,fs)\rightarrow B$ and
$\delta:cyl(s,gs)\rightarrow B$ such that $\gamma  p=f$, $\gamma
q=Id_B$, $\delta  k=g$ and $\delta  l=Id_B$. Then, the left diagram
below provides, applying $cyl$ by columns, the right
diagram below
\begin{equation}\label{diagrAux1}\xymatrix@C=15pt@R=15pt{
 A  \ar[r]^-{f} & B  & B  \ar[l]_-{Id}                  &&   A  \ar[r]^-{f} \ar[d]_{i_A} & B\ar[d]  & B \ar[d]^{i_B} \ar[l]_-{Id} \\
 A \ar[d]_{Id}\ar[u]^-{Id} \ar[r]^-{p} & cyl(s,fs) \ar[u]^-{\gamma}\ar[d]_-{\alpha} & B\ar[d]^-{Id}\ar[l]_-{q}\ar[u]_-{Id}    && cyl(A) \ar[r] & cyl(\gamma,\alpha)  & cyl(B)\ar[l]_-e\\
 A \ar[r] & cyl(s\nu,H)   & B\ar[l]_-w                 &&  A \ar[r]\ar[u]^{j_A} & cyl(s\nu,H)  \ar[u] & B\ar[l]_-w \ar[u]_{j_B}}
\end{equation}
where $e$ is an equivalence by the 2-out-of-3 property. Hence, composing maps in diagrams (\ref{diagrAux0}) and (\ref{diagrAux1}), and annexing $\delta$, we get the left diagram below, which gives rise again to the diagram on the right hand side
$$\xymatrix@C=15pt@R=15pt{
 cyl(A) \ar[r] & cyl(\gamma,\alpha)     & cyl(B)\ar[l]_-e            && cyl(A) \ar[d] \ar[r] & cyl(\gamma,\alpha)\ar[d]_-{\epsilon}     & cyl(B)\ar[l]_-e \ar[d]^{j'_B}  \\
 A \ar[u]^-{j_A} \ar[r]^-{k}\ar[d]_{Id} & cyl(s,gs) \ar[u]^{\beta'}\ar[d]_-{\delta} & B\ar[u]_-{j_B}\ar[d]^-{Id} \ar[l]_-{l}       && cyl(Id,j_A)  \ar[r]    & cyl(\delta,\beta')   & cyl(Id,j_B)\ar[l]_-{t}\\
 A  \ar[r]^-{g}                         & B    & B\ar[l]_-{Id}    && A  \ar[r]^-{g}\ar[u]   & B  \ar[u] & B\ar[l]_-{Id}\ar[u]_{I}  \ . }
$$
Note that $\epsilon$ and $j'_B$ are in $\mrm{E}$ since $\delta$
and $Id_B$ are. As $e\in\mrm{E}$, then $t\in\mrm{E}$. Composing
with (\ref{diagrAux1}) we deduce the commutative diagram in
$\mc{D}$
$$\xymatrix@C=15pt@R=15pt{
 A \ar[d] \ar[r]^f      & B\ar[d]_-{\epsilon'}     & B\ar[l]_-{Id} \ar[d]^{J}  \\
 cyl(Id,j_A)  \ar[r]^-{H'}    & cyl(\delta,\beta')   & cyl(Id,j_B)\ar[l]_-{t}\\
 A  \ar[r]^-{g}\ar[u]   & B  \ar[u]^{} & B\ar[l]_-{Id}\ar[u]_{I}   }$$
By construction, the left and right columns are part of admissible
cylinders of $A$ and $B$ respectively. Hence $t J$ and $t I$ are homotopic maps, which are equivalences. Write
$t'=[t J]=[t I]:B\rightarrow B'=cyl(\delta,\beta')$ in $K\mc{D}$. To
finish, note that $H'$ is a homotopy between $t J f$ and $t I g$, so $t' f=t' g$ in $K\mc{D}$.
\end{proof}

\begin{obs}[On the existence of $Ho\mc{D}$] If $\mc{D}$ is $\mc{U}$-small, for a fixed universe $\mc{U}$, the category
$Ho\mc{D}$ may not be $\mc{U}$-small. For, recall that the
category of (unbounded) complexes over any abelian category
$\mc{A}$ (with countable sums) is a simplicial descent category, where $\mrm{E}$=quasi-isomorphisms. But there are cocomplete
abelian categories whose associated
derived category $Ho\mc{D}$ does not have small hom's, as the following example due to P. Freyd. If $R$ is the free polynomial ring on the class consisting of all cardinals, then the abelian category $\mc{A}$ of small $R$-modules satisfies the property (AB 5), but $Hom_{D(\mc{A})}(\mathbb{Z},\Sigma\mathbb{Z})\equiv \mrm{Ext}^1_{\mc{A}}(\mathbb{Z},\mathbb{Z})$ is a proper class \cite{CN}.\\
On the other hand, when the class $\mrm{E}$ is a `locally
small multiplicative system' \cite{W}, then $Ho\mc{D}$ is $\mc{U}$-small. This is the situation, for instance, when the simplicial descent structure comes from a Quillen model structure, as the case of fibrant spectra.
\end{obs}

Next result is a direct consequence of the previous description of the morphisms in $Ho\mc{D}$, taking into account that $f\sqcup g = f'\sqcup g'$ in $Ho\mc{D}$ whenever $f$ is homotopic to $f'$ and $g$ is homotopic to $g'$ in $\mc{D}$.

\begin{cor}\label{HoDCopro} The homotopy category $Ho\mc{D}$ has (finite) coproducts, which are preserved by the canonical functor $\delta:\mc{D}\rightarrow Ho\mc{D}$.
\end{cor}


\section{Cofiber sequences}\label{Cofsequenc}

In this section we
consider the cofiber sequences in $Ho\mc{D}$ induced by
the simplicial descent structure, and prove they satisfy the `axioms' of
a (not necessarily pointed) triangulated category.

\begin{defi}
The \textit{cone} functor $c:Maps(\mc{D})\rightarrow
\mc{D}$ is defined as
$c(f)=cyl(\xymatrix@H=4pt@M=4pt@C=15pt@R=15pt{\ast
&A \ar[r]^f\ar[l] & B})$. It is equipped with a natural map
$i:B\rightarrow c(f)$.\\
The \textit{suspension} (or \textit{shift}) functor
${\Sigma}:\mc{D}\rightarrow\mc{D}$ is
${\Sigma}A=cyl(\ast\leftarrow
A\rightarrow\ast)=c(A\rightarrow\ast)$. By lemma \ref{C3}, $\Sigma:\mc{D}\rightarrow\mc{D}$ sends
equivalences to equivalences, hence it induces
$\Sigma:Ho\mc{D}\rightarrow Ho\mc{D}$.
\end{defi}

\begin{defi} We consider triangles defined
through the suspension ${\Sigma}:Ho\mc{D}\rightarrow Ho\mc{D}$, that is,
sequences $A\rightarrow B \rightarrow C\rightarrow {\Sigma}A$. A
\textit{cofiber sequence} in $Ho\mc{D}$ is a triangle
isomorphic (in $Ho\mc{D}$) to a triangle
\begin{equation}\label{distingTriangl}
\xymatrix@H=4pt@M=4pt@C=15pt@R=15pt{A \ar[r]^-f & B \ar[r]^-i & c(f)
\ar[r]^-p & {\Sigma}A}
\end{equation}
for some morphism $f$ of $\mc{D}$. The map $p:c(f)\rightarrow
{\Sigma}A$ is the cone of the canonical morphism in $Maps(\mc{D})$
from $A\stackrel{f}{\rightarrow} B$ to $A\rightarrow \ast$.

A cofiber sequence can be extended to an infinite sequence of the form
$$\xymatrix@H=4pt@M=4pt@C=15pt@R=15pt{A \ar[r]^-u & B \ar[r]^-v & C
\ar[r]^-w & {\Sigma A}\ar[r]^{\Sigma u} & \Sigma B\ar[r]^{\Sigma v} & \cdots}
$$
Note that in (\ref{distingTriangl}) $p i$ factors through $\ast$ in $\mc{D}$ by definition. The same holds for $i f$, this time up to homotopy. Hence any
composition of consecutive morphisms in a cofiber sequence is
trivial in $Ho\mc{D}$, that is, it factors through $\ast$.
\end{defi}

\begin{ejs} Classical examples of such cofiber sequences are those of topological spaces, simplicial sets and chain complexes. All of them come from their respective simplicial structures \cite{R1}.\\
\indent \textbf{{Simplicial sets:}} The simple functor is $\mrm{D}:\simp\simp Set\rightarrow \simp Set$, the diagonal functor of a bisimplicial set. Then the obtained cylinder functor is the classical one, that is $cyl(A)=A\times\Dl[1]$, as well as the induced cofiber sequences. A detailed study of them can be found, for instance, in \cite{GZ}.\\
\indent {\textbf{Topological Spaces:}} Given a continuous map $f$ of topological spaces, the cone $c(f)$ of $f$ induced by the descent structure  is the `fat geometric realization' \cite[Appendix A]{S} of $C(f)$, the simplicial cone of $f$. Since all degeneracy maps of $C(f)$ are closed cofibrations (they are coproducts of identities), then $c(f)$ is naturally homotopic to the classical cone associated with $f$.\\
\indent \textbf{Chain complexes}
If $\mc{A}$ is an additive or abelian category with countable coproducts, then a simple functor $\mbf{s}:\simp\mc{D}\rightarrow\mc{D}$ can be defined through the total complex of the (eventually normalized) double complex given by a simplicial chain complex.\\
The cone and cylinder objects induced by the descent structure are the usual one. Therefore the resulting cofiber sequences are the usual `distinguished triangles' in the associated homotopy category if $\mrm{E}$=homotopy equivalences, or in the derived category of $\mc{A}$ if $\mc{A}$ is abelian and $\mrm{E}$=quasi-isomorphisms.\\
If we restrict our attention to uniformly bounded chain complexes, for instance positive chain complexes, then the simple $\mbf{s}$ can always be defined, with no assumption on countable coproducts. On the other hand, if one is only interested in the properties of cofiber sequences of unbounded chain complexes, this assumption could be dropped as well. The idea is to consider a notion of `finite simplicial descent category'. It would involve just `finite' simplicial objects (which are degenerate in dimensions greater than a fixed integer), but we do not go into this task here. The reason why this restriction would work is that cofiber sequences are defined through the action $-\boxtimes K$, where $K$ is degenerate in dimensions greater than 2.
\end{ejs}


We return now to a general simplicial descent category $\mc{D}$. A relevant property of the cofiber sequences in $Ho\mc{D}$ is that they include the image under the simple functor of the `classical' cofiber sequences in $\simp\mc{D}$.

\begin{lema}\label{CofSeqDlD} A map $F:A\rightarrow B$ in $\simp\mc{D}$ gives rise to the sequence
$A\stackrel{F}{\rightarrow} B\stackrel{I}{\rightarrow} C(F) \stackrel{P}{\rightarrow}\Lambda A$,
where $C(F)$ is the $($simplicial$)$ cone of $F$ and $\Lambda A=C(A\rightarrow \ast)$. Then $\mbf{s}A\stackrel{\mbf{s}F}{\rightarrow} \mbf{s}B\stackrel{\mbf{s}I}{\rightarrow} \mbf{s}C(F) \stackrel{\mbf{s}P}{\rightarrow} \mbf{s}\Lambda A\equiv\Sigma \mbf{s}A$ is a cofiber sequence in $Ho\mc{D}$.
\end{lema}

\begin{proof} Lemma \ref{sCyls} provides a natural isomorphism of $Ho\mc{D}$ between $\mbf{s}CF=\mbf{s}Cyl(\ast\leftarrow A\stackrel{F}{\rightarrow} B)$ and $L=\mbf{s}Cyl(\mrm{R}\ast\leftarrow \mbf{s}A\stackrel{\mbf{s}F}{\rightarrow}\mbf{s}B)$. This isomorphism is induced by the transformation $\mu$ relating the iterated simple of certain bisimplicial object with the simple of its diagonal. On the other hand, the equivalence $\lambda_\ast:\mrm{R}\ast\rightarrow \ast$ produces an equivalence $L\rightarrow \mbf{s}C(\mbf{s}F\times\Dl)=c(\mbf{s}F)$. Putting all together we obtain a natural isomorphism $\omega:\mbf{s}CF\rightarrow c(\mbf{s}F)$ of $Ho\mc{D}$ with $\omega \mbf{s}I=i$ (here we are using the compatibility conditions (\ref{compatibLambdaMuEquac}) between $\lambda$ and $\mu$). We can construct analogously $\omega':\mbf{s}\Lambda A\rightarrow \Sigma \mbf{s}A$ with $\omega' \mbf{s}P= p\omega$. Hence the simple of our given sequence is isomorphic to the cofiber sequence defined by $\mbf{s}F$.
\end{proof}

The rest of the section is devoted to prove the following

\begin{thm}\label{struct triangulada} Let $\mc{D}$ be a simplicial descent category. Then, the cofiber sequences $($\ref{distingTriangl}$)$ of $Ho\mc{D}$ satisfy axioms TR1, TR3 and TR4 $($octahedron$)$ of triangulated categories, together with a
`non-stable' version of \emph{TR 2} given in proposition \ref{TR2}.
In addition, a descent functor $\psi:\mc{D}\rightarrow \mc{D}'$ preserves cofiber sequences.
\end{thm}

\begin{proof} Let us see the last statement. Let $\psi:\mc{D}\rightarrow \mc{D}'$ be a descent functor. Since $\psi$ commutes with coproducts, the natural map $Cyl(\psi f, \psi g)\rightarrow \psi Cyl(f,g)$ is a degree-wise equivalence. Applying $\mbf{s}'$ and composing with $\Theta$ yields a natural isomorphism $\theta:\mbf{s}'Cyl(\psi f,\psi g)\rightarrow \psi\mbf{s}Cyl(f,g)$ of $Ho\mc{D}$. As $\psi$ preserves final objects up to equivalence, $\psi(\Sigma A)\simeq \Sigma (\psi A)$ and $\psi(c(f))\simeq c(\psi f)$, so $\psi$ preserves cofiber sequences. The rest of the theorem will be proven below.
\end{proof}

\begin{prop}[\textbf{TR 1}]\mbox{}\\
i) There is a cofiber sequence
$A\stackrel{Id}{\rightarrow}A\rightarrow \ast{\rightarrow} \Sigma A$.\\
ii) Every triangle isomorphic to a cofiber sequence is so.\\
iii) Every map $f$ of $Ho\mc{D}$ is part of a cofiber sequence
$\xymatrix@C=13pt{A\ar[r]^f
&B\ar[r]&C\ar[r]&\Sigma A}$.
\end{prop}

\begin{proof}
\textit{i}) follows from lemma \ref{C2}, while \textit{ii}) holds by
definition. Let us prove \textit{iii}). By theorem \ref{MorfHoD},
$f:A\rightarrow B$ is represented by a zig-zag
$\xymatrix@M=4pt@H=4pt@C=15pt@R=15pt{A \ar[r]^{\overline{f}} &  T&
\ar[l]_w  B }$.  Then
$\xymatrix@M=4pt@H=4pt@C=15pt{A\ar[r]^{f} &  B \ar[r]^{i w^{-1}} &
c(\overline{f})\ar[r]^-{p} & \Sigma A}$ is clearly a cofiber sequence.
\end{proof}
\begin{prop}[\textbf{TR 3}]\label{TR3}
Consider the diagram of solid arrows
$$\xymatrix@M=4pt@H=4pt@R=15pt@C=15pt{A\ar[r]^f \ar[d]_{\alpha} & B \ar[r]\ar[d]_{\beta}    & C \ar[r]\ar@{.>}[d]_{\gamma} & \Sigma A\ar[d]_{\Sigma \alpha}\\
                         A'\ar[r]^{g}                 & B'\ar[r]   & C' \ar[r]& \Sigma A'.}$$
where the rows are cofiber sequences and the left square
commutes in $Ho\mc{D}$. Then $(\alpha,\beta)$ extends to a
morphism of triangles $(\alpha,\beta,\gamma)$.
\end{prop}

\begin{proof}
We can assume that $f$ and $g$ are morphisms of $\mc{D}$ and
$C=c(f)$, $C'=c(g)$. If $\alpha$ and $\beta$ are morphisms in
$\mc{D}$ and $\beta f = g\alpha$ in $\mc{D}$, then we
set $\gamma={c}(\alpha,\beta)$. Note that if $\alpha,\beta$ are
equivalences then so is $\gamma$ by lemma \ref{C3}.\\
Suppose now that $\alpha$, $\beta$ are morphisms of $\mc{D}$ and $\beta
f=g\alpha$ up to a homotopy $H:\overline{A}\rightarrow B'$. We can
obtain $\gamma$ from the previous case, applying the cone functor
by columns to
$$\xymatrix@M=4pt@H=4pt@C=15pt@R=15pt{
 A\ar[d]_{f}\ar[r]^{\sim} & \overline{A}  \ar[d]_{H}  & A\ar[l]_{\sim}\ar[d]_{g\alpha} \ar[r]^-{\alpha} &  A'\ar[d]_{g}\\
 B \ar[r]^{\beta}  &  B'                    & B'\ar[l]_{Id} \ar[r]^{Id}                &  B'.}$$
The general case is deduced from this case as usual, using that
$\mrm{E}$ has calculus of left fractions in $K\mc{D}$.
\end{proof}

Now we will study the octahedron axiom {TR 4}.
Two composable morphisms
$A\stackrel{u}{\rightarrow}B\stackrel{v}{\rightarrow}C$ in
$\mc{D}$ give rise in a natural way to the triangle
\begin{equation}\label{TriangOctaedro}c(u)\stackrel{\alpha}{\rightarrow} c(v u)\stackrel{\beta}{\rightarrow} c(v)\stackrel{\gamma}{\rightarrow} \Sigma c(u)\end{equation}
where $\alpha$ and $\beta$ come from the squares
$$\xymatrix@M=4pt@H=4pt@R=15pt@C=15pt{ A \ar[d]_{Id}\ar[r]^u          & B \ar[d]^{v}     && A\ar[r]^{v u} \ar[d]_{u} & C\ar[d]^{Id} \\
                         A \ar[r]^{v u}                          & C\ar@{}[rru]|{;} && B\ar[r]^{v}     & C           .}$$
On the other hand, $c(v)\stackrel{\gamma}{\rightarrow}\Sigma c(u)$ is the
composition $c(v)\stackrel{p}{\rightarrow} \Sigma B
\stackrel{\Sigma p}{\rightarrow} \Sigma c(u)$.

\begin{prop}[\textbf{TR 4}]
The sequence $(\ref{TriangOctaedro})$ is a cofiber sequence.
Moreover, Verdier's octahedron axiom holds in $Ho\mc{D}$.
\end{prop}

\begin{proof} Note that $(\ref{TriangOctaedro})$ is the image under the simple functor of the sequence $C(u)\stackrel{\alpha'}{\rightarrow} C(vu) \stackrel{\beta'}{\rightarrow} C(v) \stackrel{\gamma'}{\rightarrow} \Lambda C(u)$, where $\mbf{s}\Lambda C(u)$ is identified with $\mbf{s}\Lambda\mbf{s}C(u)=\Sigma c(u)$ (by lemma \ref{sCyls}). By lemma \ref{CofSeqDlD}, it is enough to see that $(\ref{TriangOctaedro})$ is isomorphic to the image under the simple functor of $C(u)\stackrel{\alpha'}{\rightarrow} C(vu) \stackrel{I}{\rightarrow} C(\alpha') \stackrel{P}{\rightarrow} \Lambda C(u)$.\\
Define $\psi:C(v)\rightarrow C(\alpha')$ by sending $B\rightarrow C(u)$ and $C\rightarrow C(vu)$. Then $\psi$ fits into the diagram of simplicial objects
$$\xymatrix@M=4pt@H=4pt@R=12pt@C=12pt{
 C(u)\ar[r]^-{\alpha'}\ar[d]_{Id} &  C(v u)\ar[d]_{Id} \ar[r]^-{\beta'}         & C(v) \ar[r]^-{\gamma'} \ar[d]_{\psi}& \Lambda C(u)\ar[d]^{Id}\\
 C(u)\ar[r]^-{\alpha'} &  C(v u) \ar[r]^{I}\ar@{}[ru]|{\mathrm{(1)}} & C(\alpha') \ar[r]     & \Lambda C(u)\ar@{}[lu]|{\mathrm{(2)}}}
$$
and square $\mathrm{(2)}$ is commutative by definition. It remains to see that $\mbf{s}\psi$ is an equivalence and that the simple of square $\mathrm{(1)}$ commutes in $Ho\mc{D}$.\\
Let us see first that $\mbf{s}\psi$ is an equivalence. Denote by $\overline{\alpha}:CA\rightarrow C(v)$ the map induced by $u$ and $v$. By lemma \ref{CilindroIterado}, reordering the coproduct $C(\alpha')_n$ yields an isomorphism $\Theta:C(\alpha')\rightarrow C(\overline{\alpha})$ such that $\Theta\psi$ is the canonical map $I_{C(v)}:C(v)\rightarrow C(\overline{\alpha})$. But $\mbf{s}(C(v)\rightarrow C(\overline{\alpha}))$ is equivalent to the map $i_{c(v)}:c(v)\rightarrow c(c(A)\rightarrow c(v))$, which is an equivalence by lemma \ref{C2}. Therefore $\mbf{s}\psi\in\mrm{E}$ by the 2-out-of-3 property.\\
Let us check now that $\mbf{s}\mathrm{(1)}$ commutes in $Ho\mc{D}$. We will see that there exists $\varphi:C(\overline{\alpha})\rightarrow C(v)$ with $\mbf{s}\varphi\in\mrm{E}$ and
$\varphi\Theta\psi\beta'=\varphi\Theta I$ in $\simp\mc{D}$.
Note that the square
$$\xymatrix@M=4pt@H=4pt@R=12pt@C=12pt{
 CA\ar[r]^-{I_{CA}}\ar[d]_-{\overline{\alpha}} & CCA\ar[d] \\
 C(v) \ar[r]^-{I_{C(v)}} & C(\overline{\alpha})}$$
is a pushout square. We have the maps $I_{CA},\varrho:CA\rightarrow CCA$ where $I_{CA}$ is the canonical map and $\varrho$ is induced by $(I_A:A\rightarrow CA,I_A:A\rightarrow CA)$. We state that there exists $\tau:CCA\rightarrow CA$ with $\tau I_{CA}=\tau \varrho=Id_{CA}$ in $\simp\mc{D}$.\\
To this end, consider the retraction $r:\Dl[1]\times\Dl[1]\rightarrow\Dl[1]$ such that $r (d^0\times Id)=r(Id\times d^0)=Id$ and the maps $r(d^1\times Id)$, $r(Id\times d^1)$ factor through $d^1:\Dl[0]\rightarrow \Dl[1]$. For, if $s,t:[n]\rightarrow [1]$ and $i\in [n]$, then $[r(s,t)](i)=1$ if $s(i)=t(i)=1$ and $[r(s,t)](i)=0$ otherwise.\\
On the other hand, the cone $CA$ of an object $A$ is defined as the pushout of the map $d^0\sqcup d^1:A\sqcup A\rightarrow A\boxtimes \Dl[1]$ along $A\sqcup A\rightarrow A\sqcup \ast$. Hence $CCA$ is defined by a double pushout involving $A\boxtimes(\Dl[1]\times\Dl[1])$, and the desired map $\tau$ is obtain from $Id_A\boxtimes r:A\boxtimes(\Dl[1]\times\Dl[1])\rightarrow A\boxtimes\Dl[1]$.\\
Now, the maps $\overline{\alpha}\tau:CCA\rightarrow C(v)$ and $Id:C(v)\rightarrow C(v)$ induce a map $\varphi:C(\overline{\alpha})\rightarrow C(v)$, and it is straightforward to check the equalities $\varphi\Theta\psi\beta'=\varphi\Theta I$.\\
To finish the proof, any morphism in $Ho\mc{D}$ is represented by a zig-zag $A\rightarrow T\stackrel{e}{\leftarrow}B$, with $e\in\mrm{E}$. Then the previous case, together with \htmlref{{TR 3}}{TR3}, implies the octahedron in the general case.
\end{proof}

Next we define a generalization of
$-Id:\Sigma B\rightarrow \Sigma B$, also valid in the non-additive (even non-pointed) case. It comes from the one existing at the simplicial level \cite[p. 47]{Vo}.

Let $B$ be an object of $\mc{D}$. In $\simp\mc{D}$ we have the map $I_B:B\times\Dl\rightarrow CB$, where $CB$ denotes de cone of the identity of $B\times\Dl$. Set $\Lambda_2^1 B=C(I_B)\in\simp\mc{D}$ and $\Lambda B=C(B\rightarrow \ast)$. Consider the maps $P_1,P_2:\Lambda_2^1 B\rightarrow \Lambda B$ defined as follows. On one hand, $P_1$ sends $CB$ to $\ast$, and $B$ to $B$; on the other, $P_2$ sends $C(B)$ to $\Lambda B$, and $B$ to $\ast$.

\begin{defi}\label{signoMenos}
Set $\Sigma^1_2B=\mbf{s}\Lambda_2^1 B$, and $p_i=\mbf{s}P_i:\Sigma^1_2B\rightarrow \Sigma B$, $i=1,2$. As $c(B)\rightarrow \ast$ is an equivalence, then $p_1$ is so by lemma \ref{C3}. Define
$m_B:\Sigma B\rightarrow \Sigma B$ as the composition
$$ \xymatrix@M=4pt@H=4pt@R=15pt@C=15pt{\Sigma B \ar[r]^{p_1^{-1}}& \Sigma^1_2B\ar[r]^{p_2} & \Sigma B}.$$
\end{defi}

\begin{lema}\label{lemaSignoMenos}
The map $m_{-}$ is a natural transformation, that is, $m_{B'} {\Sigma f}={\Sigma f} m_B:\Sigma B\rightarrow \Sigma B'$ for any $f:B\rightarrow B'$ in $Ho\mc{D}$. Also, $m_B$ is an isomorphism of $Ho\mc{D}$ such that $m_B^2=Id$.
\end{lema}

\begin{proof} The first assertion follows trivially from the definition of $m$. Let us see that $m_B^2:\Sigma B\rightarrow \Sigma B$ is the identity in $Ho\mc{D}$. By lemma \ref{CilindroIterado}, the interchange of the terms in the coproduct $(\Lambda^1_2 B)_n$ yields an automorphism $\Theta:\Lambda^1_2 B\rightarrow \Lambda^1_2 B$ of $\simp\mc{D}$ such that $\Theta P_1 =P_2$. But $\Theta^2=Id$, so $m_B^2=Id$.
\end{proof}

\begin{prop}[\textbf{TR 2}]\label{TR2}
If
$A\stackrel{u}{\rightarrow} B\stackrel{v}{\rightarrow}
C\stackrel{w}{\rightarrow} \Sigma A$ is a cofiber sequence,
then so is
\begin{equation}\label{AuxTR2}\xymatrix@M=4pt@H=4pt@R=15pt@C=19pt{B\ar[r]^-{v}&  C\ar[r]^-{w} &  \Sigma A \ar[r]^-{m_B \Sigma u} & \Sigma B}.\end{equation}
\end{prop}

\begin{proof} We can assume that our given cofiber sequence is induced by
a map $u:A\rightarrow B$ in $\mc{D}$. That is, $C=c(u)$, $v=i:B\rightarrow c(u)$ and $w=p:c(u)\rightarrow \Sigma A$. Let us prove that (\ref{AuxTR2}) is isomorphic (in $Ho\mc{D}$) to the simple of the sequence $B\stackrel{I}{\rightarrow} C(u)\stackrel{I'}{\rightarrow} C(I)\stackrel{P_B}{\rightarrow} \Lambda B$. We will find a map $\psi:C(I)\rightarrow \Lambda A$ with $\mbf{s}\psi\in\mrm{E}$ and such that the simple of diagram
$$\xymatrix@M=4pt@H=4pt@R=3pt@C=6pt{
 B\ar[rr]^-{I}\ar[dd]_{Id} &&  C(u)\ar[dd]_{Id} \ar[rr]^-{I'}         && C(I) \ar[rr]^-{P_B} \ar[dd]_{\phi}&&  \Lambda B & \\
              &&                    &&                                   && &  \Lambda^2_1 B \ar[lu]_{P_1}\ar[ld]^{P_2}\\
 B\ar[rr]^-{I} &&  C(u) \ar[rr]^{P_A} && \Lambda A \ar[rr]^-{\Lambda u}     && \Lambda B & }
$$
commutes in $Ho\mc{D}$.
The maps $P_A:C(u)\rightarrow \Lambda A$ and $I_\ast:\ast\rightarrow \Lambda A$ are such that $P_A I$ is equal to $I_\ast$ composed with the trivial map $B\rightarrow\ast$. Then, they produce $\phi:C(I)\rightarrow \Lambda A$
with $\phi I'=P_A$. On the other hand, $u$ induces a map $u':C(u)\rightarrow CB$ with $u'I=I_B:B\rightarrow CB$. Then, $u'$ induces $\overline{u}:C(I)\rightarrow \Lambda_2^1 B=C(I_B)$. It follows from definitions that $P_2\overline{u}=\Lambda u \phi$ and $P_1\overline{u}=P_B$. Therefore, the simple of the above diagram commutes in $Ho\mc{D}$. As $m_B^2=Id$, then $m_B=\mbf{s}P_2^{-1}\mbf{s}P_1$ and $\mbf{s}\phi$ is part of a morphism of triangles.\\
To finish it remains to see that $\mbf{s}\phi\in\mrm{E}$. For, by lemma \ref{CilindroIterado}, $C(I)$ is isomorphic to the cone of the map $w:A\rightarrow CB$ induced by $u$, and $\phi$ corresponds to the map $C(w)\rightarrow \Lambda A$ that sends $CB$ to $\ast$. But this is an equivalence by lemma \ref{C3}.
\end{proof}


Recall that the dual of a simplicial descent category is by definition a cosimplicial descent category. In this case, a \textit{path}
functor in $\mc{D}$ is induced by $path (f,g)=\mbf{s}Path (f,g)$. More concretely, if $X\in\simp\mc{D}$ and $K$ is a pointed simplicial finite set, then $X^{K}$ is in degree $n$ equal to $\prod_{K_n}X^n$. If $C\stackrel{g}{\rightarrow}A\stackrel{f}{\leftarrow}B$ are maps in $\mc{D}$, then $Path(f,g)$ is the pullback of $A^{\Dl[1]}\stackrel{d^0\times d^1}{\rightarrow}A\times A$ and $B\times C\stackrel{f\times g}{\rightarrow} A\times A$.

The dual to the suspension is then the \textit{loop} functor
$\Omega:\mc{D}\rightarrow\mc{D}$, defined as
$\Omega(X)=path(0\rightarrow X \leftarrow 0)$, and this time the
fiber sequences considered are
$$ \Omega X \rightarrow path(f) \rightarrow X\stackrel{f}{\rightarrow} Y$$
All results concerning cofiber sequences can be dualized for fiber sequences, inducing a `left' triangulated structure on $Ho\mc{D}$.

\begin{ejs}\mbox{}\\
\indent \textbf{Differential graded algebras:} Let $\mbf{Dga}(R)$ be the category of differential graded $R$-algebras (not necessarily commutative, and positively graded) over a commutative ring $R$. The (normalized) Alexander-Whitney simple $\mbf{s}_{AW}:\simp\mbf{Dga}(R)\rightarrow \mbf{Dga}(R)$ comes from the normalized simple of cochain complexes of $R$-modules and endows $\mbf{Dga}(R)$ with a structure of cosimplicial descent category, where $\mrm{E}$=quasi-isomorphisms \cite{R1}.\\
If $f:B\rightarrow A$ and $g:C\rightarrow A$ are morphisms of differential graded algebras, then $path_{dga}(f,g)$ has as underlying cochain complex the object  $path(f,g)$ in $Ch (R-mod)$, while the product in $path_{dga}(f,g)$ is induced by those of $A$, $B$ and $C$ (see \cite{G} for details). The deduced fiber sequences are studied in loc. cit. as well, where an analogous result to \ref{TR2} is proven.\\
\indent \textbf{Commutative differential graded algebras:} Denote by $\mbf{Cdga}(k)$ the category of commutative dif{f}erential
graded algebras over a field $k$ of characteristic 0.
Navarro's Thom-Whitney simple
\cite{N} $\mbf{s}_{TW}:\simp\mbf{Cdga}(k)\rightarrow\mbf{Cdga}(k)$ gives rise to a cosimplicial descent
structure on $\mbf{Cdga}(k)$, where $\mrm{E}=$quasi-isomorphisms \cite{R1}.
We deduce in this case a functor $path_{cdga}$, such that the underlying cochain complex of $path_{cdga}$ is quasi-isomorphic to the usual path functor in cochain complexes. The resulting distinguished triangles verify then theorem \ref{struct triangulada}, and the forgetful functor from $\mbf{Cdga}$ to cochain complexes of $k$-vector spaces sends a cofiber sequence to a distinguished triangle of $D(k-mod)$.
\end{ejs}

If the category $\mc{D}$ is an additive simplicial descent category, we deduce a \textit{suspended}
(or right triangulated) category structure \cite{KV} on $Ho\mc{D}$. We will prove later that $Ho\mc{D}$ is always additive for a (not necessarily additive) stable simplicial descent category $\mc{D}$. Stable means that the suspension is an equivalence of categories. But the additive and non stable case is still interesting. It covers, for instance, the case $\mrm{CF}^c\mc{A}$ of uniformly bounded-below (regularly) filtered cochain complexes, described later. The induced suspended structure gives rise to a Verdier's triangulated structure on the homotopy category of bounded-below filtered complexes. It is not obtained directly as the homotopy category of a simplicial descent category. The reason is that the simple $\mbf{s}:\Dl\mrm{CF}^c\mc{A}\rightarrow \mrm{CF}^c\mc{A}$ does not preserve regular filtrations in the (non-uniformly) bounded-below case.

\begin{defi}\cite{R1}
An \textit{additive} simplicial descent category is by definition
an additive category $\mc{D}$ endowed with a simplicial descent structure $(\mc{D},\mrm{E},\mbf{s},\mu,\lambda)$ such that $\mbf{s}$ is an additive functor and $\mu$ is an
isomorphism in $Fun_{ad}(\simp\simp\mc{D},\mc{D})[\mrm{E}^{-1}]$.
Here $Fun_{ad}(\simp\simp\mc{D},\mc{D})$ is the category of
additive functors from $\simp\simp\mc{D}$ to $\mc{D}$.
\end{defi}

\begin{cor} If $\mc{D}$ is an additive simplicial descent category, then $Ho\mc{D}$ is a suspended category \cite{KV}. In addition, a descent functor $\mc{D}\rightarrow\mc{D}'$ induces a functor of suspended categories $F:Ho\mc{D}\rightarrow Ho\mc{D}'$, that is, it preserves cofiber sequences.
\end{cor}

\begin{proof} If $\mc{D}$ is additive, one easily checks that $Ho\mc{D}$ is additive as well using \ref{MorfHoD}. The suspension $\Sigma$ is additive since it is composition of additive functors. To finish it remains to see that the `abstract' minus sign $m_B:\Sigma B\rightarrow \Sigma B$ is equal to $-Id_{\Sigma B}$ in $Ho\mc{D}$. This can be done directly, finding a homotopy between $m$ and $-Id$, or it can be deduced from proposition \ref{A1cogrupo}. Indeed, we have two group structures on $Hom_{Ho\mc{D}}(\Sigma B,\Sigma B)$. The cogroup object structure on $\Sigma B$ given by the map $d_B:\Sigma B\rightarrow \Sigma B\sqcup \Sigma B=\Sigma B\oplus\Sigma B$ of \ref{A1cogrupo} induces the first one. The group object structure given by the codiagonal $\Sigma B\oplus\Sigma B=\Sigma B\times \Sigma B \rightarrow \Sigma B$ induces the second one, that is the one coming from the sum in $\mc{D}$. By standard arguments these group structures agree, so in particular $m_B=-Id_{\Sigma B}$.
\end{proof}

\begin{ejs} We exhibit the suspended structures coming from additive (and non-stable) simplicial descent categories.

\textbf{Simplicial objects in additive$/$abelian categories:} If $\mc{A}$ is an additive or abelian category, consider $\mc{D}=\simp\mc{A}$. The simple functor $\mrm{D}:\simp \simp\mc{A}\rightarrow\simp\mc{A}$ is the diagonal functor of a bisimplicial object. Hence, our cofiber sequences are just the \textit{cofibration sequences} of \cite[5.3]{Vo}. They correspond to the usual `exact triangles' in $Ch_+\mc{A}$, the category of positive chain complexes of $\mc{A}$, through the functor $K:\simp\mc{A}\rightarrow Ch_+\mc{A}$ that takes as boundary map the alternate sum of the face maps in a simplicial object. The same holds for the normalized version of $K$, which is an equivalence of categories \cite[$\S 22$]{May}.

\textbf{Filtered cochain complexes:} Given an abelian category $\mc{A}$ and an integer $c$, let $\mrm{CF}^c\mc{A}$ be the category of filtered cochain complexes $(A,\mrm{F})$, where
$A$ is a cochain complex over $\mc{A}$ that is equal to 0 in degrees lower than $c$, and $\mrm{F}$ is a decreasing biregular filtration of $A$.\\
The normalized simple $(\mbf{s}_N,\mbf{s}_N)$ comes from the one of cochain complexes, and $\mrm{E}$=filtered quasi-isomorphisms are part of an additive simplicial descent structure on $\mrm{CF}^c\mc{A}$ (see \cite{R1} for more detail). The deduced path object of the filtered maps $(C,\mrm{H})\stackrel{g}{\rightarrow}(A,\mrm{F})\stackrel{f}{\leftarrow}(B,\mrm{G})$ is the filtered cochain complex $(path(f,g),\mrm{M})$, where $path(f,g)^n=B^n\oplus A^{n-1}\oplus C^n$ and
$$\mrm{M}^kpath(f,g)^n=\mrm{G}^kB^n\oplus \mrm{F}^kA^{n-1}\oplus \mrm{H}^kC^n\ .$$
As a corollary we get the classical left triangulated (or cosuspended) structure on the filtered derived category of uniformly bounded-below complexes. Note that the category of (arbitrarily) bounded-below complexes $\mrm{CF}^b\mc{A}$ is the union of the $\mrm{CF}^c\mc{A}$ as the integer $c$ varies. The suspended structures on the filtered derived categories $\mrm{CF}^c\mc{A}$ are compatible. Therefore they induce a Verdier's triangulated structure on the usual filtered derived category $\mrm{CF}^b\mc{A}[\mrm{E}^{-1}]$.\\
We can also consider another additive simplicial descent structure on $\mrm{CF}^c\mc{A}$ \cite{R1}. The equivalences considered now are the $E_2$-isomorphisms, that is, those maps inducing isomorphism on the second term of the associated spectral sequences. In this case the simple functor is $(\mbf{s}_N,\delta)$, where $\delta$ denotes the diagonal filtration. The path object of $f$ and $g$ is $(path(f,g),\mrm{N})$, where $path(f,g)^n=B^n\oplus A^{n-1}\oplus C^n$ as before, but
$$\mrm{N}^kpath(f,g)^n=\mrm{G}^kB^n\oplus \mrm{F}^{k-1}A^{n-1}\oplus \mrm{H}^kC^n\ .$$
We get a cosuspended structure on the localized category of $\mrm{CF}^c\mc{A}$ with respect to the $E_2$-isomorphisms. Therefore, the localized category of $\mrm{CF}^c\mc{A}[E_2^{-1}]$ is a triangulated category. In addition, we deduce that Deligne's decalage functor $\mrm{Dec}:\mrm{CF}^c\mc{A}[\mrm{E}^{-1}]\rightarrow\mrm{CF}^c\mc{A}[E_2^{-1}]$ preserves this structures, since it is a functor of additive simplicial descent categories.
As before, we can induce a Verdier's triangulated structure on $\mrm{CF}^b[E_2^{-1}]$. Then,  $\mrm{Dec}:\mrm{CF}^b\mc{A}[\mrm{E}^{-1}]\rightarrow\mrm{CF}^b\mc{A}[E_2^{-1}]$ is a functor of triangulated categories.
\end{ejs}



\section{Cogroup structures in simplicial descent categories.}

In this section we deal with a simplicial descent category $\mc{D}$ such that $Ho\mc{D}$ is pointed, that is,
the map $0\rightarrow  \ast \in\mrm{E}$. In this case the following proposition holds, similarly to the case of pointed topological spaces, or more generally, of pointed model categories.

\begin{defi} Denote by $\simp {}_f Set$ the category of simplicial finite sets, and by $\simp{}_fSet_\ast$ the one of pointed simplicial finite sets. The action $\boxtimes:\simp {}_f Set\times\simp\mc{D}\rightarrow\simp\mc{D}$ induces an action $\otimes:\simp{}_f Set_\ast\times\simp\mc{D}\rightarrow\simp\mc{D}$ \cite{Vo}. Given $K\in\simp {}_f Set_\ast$ and $X\in\simp\mc{D}$, then $K\otimes X$ is the pushout of the maps $\Dl[0]\boxtimes X\rightarrow \ast$ and $\Dl[0]\boxtimes X\rightarrow K\boxtimes X$, where $\Dl[0]\rightarrow K$ is the distinguished point of $K$.\\
Recall that the coproduct in $\simp{}_fSet_\ast$ of $K$ and $L$ is $K\vee L$, the quotient of $K\sqcup L$ by the set of the  two base points of $K$ and $L$.
Then the natural map
$$(K\otimes X)\sqcup(L\otimes X)\rightarrow (K\vee L)\otimes X\ ,$$
is a degree-wise equivalence (since $\ast\sqcup \ast\rightarrow \ast$ is in $\mrm{E}$).\\
If we choose $d^1(\Dl[0])$ as the base point of $\Dl[1]$, then $X\otimes\Dl[1]$ is just $CX$, the cone of $X$. Also, if $A$ is in $\mc{D}$ then $\Sigma A$ is canonically isomorphic to $\mbf{s}(A\otimes S^1)$ in $Ho\mc{D}$, where $S^1= \Dl[1]/\mrm{sk}^0 \Dl[1]$ is the simplicial circle.
Indeed, $\ast\sqcup \ast\rightarrow\ast$ induces $\nu: \Lambda A= C(A\rightarrow\ast)\rightarrow A\otimes S^1$, so $\mbf{s}\nu :\Sigma A\rightarrow \mbf{s}(A\otimes S^1)$ is in $\mrm{E}$.

If $K\in\simp {}_f Set_\ast$, define $K\wedge \Dl[1]$ as the quotient of $K\times\Dl[1]$ by $\Dl[0]\times \Dl[1]$, where $\Dl[0]\subseteq K$ is the base-point of $K$. Then, $K\wedge\Dl\in\simp {}_f Set_\ast$ with the obvious base-point, and we have maps $d^0,d^1:K\rightarrow K\wedge\Dl[1]$ in $\simp {}_f Set_\ast$ induced by $d^0,d^1,\Dl[0]\rightarrow\Dl[1]$. Two maps $f,g:K\rightarrow L$ are \textit{simplicially homotopic} if there exists $H:K\wedge\Dl[1]\rightarrow L$ with $Hd^0=f$ and $Hd^1=g$.
\end{defi}

\begin{lema} If $f,g:K\rightarrow L$ are simplicially homotopic in $\simp {}_f Set_\ast$ and $A\in\mc{D}$, then $\mbf{s}(f\otimes A)=\mbf{s}(g\otimes A)$ in $Ho\mc{D}$.
\end{lema}

\begin{proof} Since $Ho\mc{D}$ is pointed, the natural map $Cyl(A\otimes K)\rightarrow A\otimes (K\wedge \Dl[1])$ is a degree-wise equivalence, and the statement follows from lemma \ref{relHomotopia}.
\end{proof}


\begin{prop}\label{A1cogrupo}
The correspondence $A\rightarrow Hom_{Ho\mc{D}}(\Sigma A,-)$ is a functor from $Ho\mc{D}$ to $Groups$. That is, $\Sigma A$ is a cogroup object in $Ho\mc{D}$. In addition, $\Sigma^k A$ is an abelian cogroup object if $k\geq 2$.
If $f:A\rightarrow B$ is any map in $\mc{D}$, then $\Sigma A$ coacts on $c(f)$.
\end{prop}

\begin{proof}
Define $\Omega\in\simp{}_f Set$ as the pushout
$$ \xymatrix@M=4pt@H=4pt@R=14pt@C=17pt{
 \Dl[0]\sqcup\Dl[0] \ar[r]^-{d^0\sqcup d^1}\ar[d]_{d^1\sqcup d^0} & \Dl[1]\ar[d]^p \\
 \Dl[1]\ar[r]^-q & \Omega}$$
Let $\Dl[0]\stackrel{d^1}{\rightarrow} \Dl[1]\stackrel{q}{\rightarrow} \Omega$ be the base-point of $\Omega$. Consider the map $\alpha:\Omega\rightarrow S^1$ such that $\alpha p$ is the projection $P:\Dl[1]\rightarrow S^1=\Dl[1]/\mrm{sk}^0\Dl[1]$, and $\alpha q$ is the trivial map factoring through $\Dl[0]$. If $A\in\mc{D}$, then $\Omega\otimes A$ agrees with $Cyl(CA\leftarrow A\rightarrow \ast)$ after identifying $\ast\sqcup \ast$ with $\ast$. The map $\alpha\otimes A$ sends $CA$ to $\ast$, so $\mbf{s}(\alpha\otimes A):\mbf{s}(\Omega\otimes A)\rightarrow\mbf{s}(S^1\otimes A)\equiv\Sigma A$ is an equivalence.\\
On the other hand, let $\pi:\Omega\rightarrow S^1\vee S^1$ be the map with $\pi q =i_1 P$ and $\pi p= i_2 P $, where $i_1,i_2:S^1\rightarrow S^1\vee S^1$ are the canonical inclusions. If $A\in\mc{D}$, define the map $d_A:\Sigma A\rightarrow \Sigma A\sqcup \Sigma A$ of $Ho\mc{D}$ through the composition
$$ \xymatrix@C=30pt{
\mbf{s}(S^1\otimes A)\ar[r]^-{\mbf{s}(\alpha\otimes A)^{-1}} & \mbf{s}\Omega\ar[r]^-{\mbf{s}(\pi\otimes A)} & \mbf{s}((S^1\vee S^1)\otimes A)}$$
Recall that $d_A$ endows $Hom_{Ho\mc{D}}(\Sigma A,-)$ with a group structure if and only if $d_A$ is associative, it has unit and inverse element.
Let us prove that $d_A$ has unit element, that is, that $\pi_1 d_A,\pi_2 d_A:\Sigma A\rightarrow \Sigma A$ are the identity in $Ho\mc{D}$.\\
Clearly $\pi_2 \pi =\alpha :\Omega\rightarrow S^1$, so $\pi_2 d_A=Id$. We state that $\pi_1\pi$ is simplicially homotopic to $\alpha$ in $\simp{}_f Set_\ast$. Indeed, consider the maps  $r,\widetilde{r}:\Dl[1]\otimes\Dl[1]\rightarrow\Dl[1]$ in $\simp{}_f Set$ given by $r(f,g)(i)=\mrm{max}\{f(i),g(i)\}$ and $\widetilde{r}(f,g)(i)=\mrm{min}\{f(i),g(i)\}$.
Then $r(d^0\times Id)$, $r(Id\times d^0)$, $\widetilde{r}(d^1\times Id)$ and $\widetilde{r}(Id\times d^1)$ are equal to $Id_{\Dl[1]}$. Also, $r(d^1\times Id)=r(Id\times d^1)$ factors through $d^1:\Dl[0]\rightarrow\Dl[1]$, while $\widetilde{r}(d^0\times Id)=\widetilde{r}(Id\times d^0)$ factors through $d^0$.\\
Let $H:\Omega\times\Dl[1]\rightarrow S^1$ be the map such that $H(q\times Id)=P r$ and $H(p\times Id)=P\widetilde{r}$. If $\ast\subset \Omega$ is the distinguished point of $\Omega$, then $H(\ast\times \Dl[1])$ is the base-point of $S^1$. Therefore, $H$ defines $\overline{H}:\Omega\wedge \Dl[1]\rightarrow S^1$, that is a homotopy between $\pi_1\pi$ and $\alpha$. It follows that $\pi_1 d_A=Id$ in $Ho\mc{D}$.\\
We will see next that the map $m_A:\Sigma A \rightarrow \Sigma A$ given in definition \ref{signoMenos} is indeed an inverse element for $d_A$. That is, the following composition factors through $\ast$ in $Ho\mc{D}$
\begin{equation}\label{mNeutro} \xymatrix@M=4pt@H=4pt@R=14pt@C=18pt{
  \Sigma A\ar[r]^-{d_A}& \Sigma A\sqcup \Sigma A\ar[r]^-{Id\sqcup m_A}& \Sigma A\sqcup \Sigma A \ar[r]^-{\delta_{\Sigma A}} & \Sigma A}\end{equation}
where $\delta_{\Sigma A}$ is the codiagonal. Set $\Omega A=Cyl(CA\leftarrow A \rightarrow \ast)\simeq \Omega\otimes A$, $\pi_A:\mbf{s}\Omega A\rightarrow \Sigma A\sqcup \Sigma A$ the map induced by $\mbf{s}(\pi\otimes A)$. Denote $\delta_{\Sigma A}(Id\sqcup m_A)$ by $(Id,m_A)$. By definition, the composition (\ref{mNeutro}) factors through $\ast$ if and only if $(Id,m_A)\pi_A:\mbf{s}\Omega A\rightarrow \Sigma A$ does. We will see that $(Id,m_A)\pi_A$ is isomorphic to the composition of consecutive maps in a cofiber sequence, so it is the trivial map.\\
Consider the maps $\delta'=(d^0, d^1):A\sqcup A\rightarrow Cyl(A)$ and $s^0:Cyl(A)\rightarrow A$. By lemma \ref{CilindroIterado}, $C(\delta')$ is isomorphic to $Cyl(CA\leftarrow A\rightarrow CA)$. Therefore, sending the $CA$ on the right to $\ast$ produces $f:C(\delta')\rightarrow \Omega A$ with $\mbf{s}f\in\mrm{E}$.
Consider the following diagram
$$ \xymatrix@M=4pt@H=4pt@R=12pt@C=23pt{
 \mbf{s}C(\delta') \ar[d]^{\mbf{s}f} \ar[r]^-p & (\Sigma A, \Sigma A) \ar[r]^-{\Sigma d^0\sqcup \Sigma d^1} \ar[d]_{Id\sqcup m_A} &\Sigma cyl(A)\ar[d]^{\Sigma s^0}  \\
 \mbf{s}\Omega A\ar[r]^-{\pi_A} & \Sigma A\sqcup \Sigma A \ar[r]^-{(Id,m_A)} & \Sigma A}$$
The right square commutes since $m_A^2=Id_{\Sigma A}$ and $(\Sigma d^0, \Sigma d^1)\Sigma s^0$ is the codiagonal $\delta_{\Sigma A}$. The left square is commutative as well. Indeed, $m_A=p_1^{-1}p_2=p_2^{-1}p_1$, where $p_1,p_2:\Sigma^1_2 A\rightarrow \Sigma A$ are the maps defined in \ref{signoMenos}.
There is a map $h:\mbf{s}C(\delta')\rightarrow \Sigma A \sqcup \Sigma^1_2 A  \simeq\mbf{s}Cyl(\Lambda A\leftarrow A\rightarrow CA)$ induced by $CA\rightarrow \Lambda A$. It follows from the definitions that $p_2h=p$ and $p_1h=\pi_Af$. Therefore $(Id\sqcup m_A)p=\pi_A \mbf{s}f$, and $d_A$ has inverse element.\\
It remains to check that $d_A$ is associative, that is, that $(d_A\sqcup Id)d_A=(Id_\sqcup d_A)d_A$ in $Ho\mc{D}$. It can be proved that $(Id\sqcup Id\sqcup m)(d_A\sqcup Id)d_A$ and $(Id\sqcup Id\sqcup m)(Id\sqcup d_A)d_A$ are both equal to the composition $\Sigma A\stackrel{\beta^{-1}}{\rightarrow} \mbf{s}C(\delta')\stackrel{\sigma}{\rightarrow} \Sigma A\sqcup \Sigma A\sqcup \Sigma A$, where $\beta$ sends both $CA$ in $C(\delta')$ to $\ast$, and $\sigma$ sends them to $\Lambda A$. The details are left to the reader.\\
It follows formally from the properties of $d_A$ that $\Sigma^2 A$ is an abelian cogroup object. Denote by $\tau_B:B\sqcup B\rightarrow B\sqcup B$ the isomorphism that interchanges the factors. Note that $d_{\Sigma A}\equiv\Sigma d_A:\Sigma^2A\rightarrow \Sigma^2 A\sqcup \Sigma^2 A$, since $Cyl(\Lambda f,\Lambda g)\equiv\Lambda Cyl(f,g)$ by lemma \ref{CilindroIterado}. Then, $\tau _{\Sigma^2 A} d_{\Sigma A}\equiv \Sigma (\tau_{\Sigma A}d_A)$. On the other hand,
$(\Sigma \pi_1\sqcup \Sigma \pi_2) d_{\Sigma A\sqcup \Sigma A}\equiv(\pi_1\sqcup \pi_2)(d_{\Sigma A}\sqcup d_{\Sigma A})=Id$. Then
$$\Sigma (\tau_{\Sigma A}d_A)=(\Sigma \pi_1\sqcup \Sigma \pi_2) d_{\Sigma A\sqcup \Sigma A} \Sigma (\tau_{\Sigma A}d_A)= (\Sigma \pi_1\sqcup \Sigma \pi_2)(\Sigma (\tau_{\Sigma A}d_A)\sqcup \Sigma (\tau_{\Sigma A}d_A))d_{\Sigma A}$$
The last equality holds since $d_{-}$ is a natural transformation between the functors $\Sigma$, $\Sigma \sqcup \Sigma:Ho\mc{D}\rightarrow Ho\mc{D}$. But $(\Sigma \pi_1\sqcup \Sigma \pi_2)(\Sigma (\tau_{\Sigma A}d_A)\sqcup \Sigma (\tau_{\Sigma A}d_A))=\Sigma (\pi_2 d_A)\sqcup \Sigma (\pi_1 d_A)=Id$. Therefore $d_{\Sigma A}=\tau_{\Sigma^2 A}d_{\Sigma A}$, and $\Sigma^2 A$ is abelian.\\
To finish, there is a natural coaction $w_A:cA=\mbf{s}(\Dl[1]\otimes A)\rightarrow \Sigma A \sqcup cA \simeq \mbf{s}(( S^1 \vee \Dl[1])\otimes A)$ described as follows. Consider $\overline{\Omega}\in\simp{}_f Set_\ast$ given by the pushout
$$ \xymatrix@M=4pt@H=4pt@R=14pt@C=17pt{
 \Dl[0] \ar[r]^-{ d^1}\ar[d]_{ d^0} & \Dl[1]\ar[d]^{\overline{p}} \\
 \Dl[1]\ar[r]^-{\overline{q}} & {\overline{\Omega}}  {} }$$
and with base-point $\Dl[0]\stackrel{d^1}{\rightarrow} \Dl[1]\stackrel{\overline{q}}{\rightarrow} \overline{\Omega}$. We have maps $\overline{\alpha}:\overline{\Omega}\rightarrow \Dl[1]$ and $\overline{\pi}:\Omega\rightarrow S^1\vee\Dl[1]$ with $\overline{\alpha}\,\overline{p}=Id$, $\overline{\alpha}\,\overline{q}=d^1s^0$, $\overline{\pi}\,\overline{p}=i_2$, $\overline{\pi}\,\overline{q}=i_1 P$. As before, $\mbf{s}(\overline{\alpha}\otimes A)\in\mrm{E}$. Indeed, $\mbf{s}(\overline{\Omega}\otimes A)\simeq\mbf{s}Cyl(CA\leftarrow A\rightarrow A)$ and $\mbf{s}(\overline{\alpha}\otimes A)$ comes from $CA\rightarrow \ast$. Then, $w_A$ is defined as $\mbf{s}(\overline{\pi}\otimes A)\mbf{s}(\overline{\alpha}\otimes A)^{-1}$. By definition $\pi_2 w_A=Id_{cA}$, while the proof of the equality $(Id\sqcup w_A)w_A=(d_A\sqcup Id)w_A$ is analogous to the one of the associativity of $d_A$.
If $f:A\rightarrow B$, then $\omega_A$ extends to a coaction $w_f: c(f)\rightarrow \Sigma A \sqcup c(f)$, since $C(f)$ is the pushout in $\simp\mc{D}$ of $A\rightarrow CA$ along $f:A\rightarrow B$.
\end{proof}

\begin{obs} In \cite[p. 75]{GZ}, the author proves that $S^1$
is a cogroup object in the category $\simp Set_\ast$ modulo homotopy, localized by anodyne extensions. The corresponding map $\varphi:S^1\rightarrow S^1\vee S^1$ is described by means of $\Lambda^1[2]\hookrightarrow\Dl[2]$. We give here an alternative description of $\varphi$, better suited for our class of equivalences.
\end{obs}

\begin{obs} We can redefine cofiber sequences (\ref{distingTriangl}), by forgetting $p:c(f)\rightarrow \Sigma A$ and considering instead the previous coaction of $\Sigma A$ over $c(f)$. It can be proved that the reformulations of TR1 to TR4 by means of coactions still hold for this new cofiber sequences, but we do not go into this task here.
\end{obs}


\section{Stable simplicial descent categories.}

\begin{defi} A simplicial descent category $(\mc{D},\mrm{E},\mbf{s},\mu,\lambda)$ is called \textit{stable} if the induced suspension functor $\Sigma:Ho\mc{D}\rightarrow Ho\mc{D}$ is an equivalence of categories.
\end{defi}

\begin{prop} If $\mc{D}$ is a stable simplicial descent category then $Ho\mc{D}$ is an additive category.
\end{prop}

\begin{proof} Let us see first that stability implies that $Ho\mc{D}$ is pointed. Note that the initial and final objects of $\mc{D}$ are so in $Ho\mc{D}$ by proposition \ref{MorfHoD}. As $\Sigma$ is an equivalence of categories, then $\Sigma 0$ should be isomorphic to $0$. But $\Sigma 0$ is $\ast\sqcup \ast$ by definition. Composing with an inclusion $\ast\rightarrow \ast \sqcup\ast$ we get a map $\ast\rightarrow 0$ in $Ho\mc{D}$. As compositions $0\rightarrow \ast\rightarrow 0$ and $\ast\rightarrow 0\rightarrow\ast$ are identities then $0\equiv\ast$ in $Ho\mc{D}$.\\
Since $\Sigma$ is an equivalence, it follows from proposition \ref{A1cogrupo} that each object in $Ho\mc{D}$ is an (abelian) cogroup object in a natural way. But this formally implies that $Ho\mc{D}$ is additive. Indeed, the sum of two morphisms $\Sigma f,\Sigma g:\Sigma A\rightarrow \Sigma B$ is
$$\Sigma A \stackrel{d_A}{\rightarrow} \Sigma A\sqcup\Sigma A\stackrel{\Sigma f\sqcup \Sigma g}{\longrightarrow} \Sigma B\sqcup\Sigma B\stackrel{\varsigma}{\rightarrow} \Sigma B$$
where $\varsigma$ denotes the codiagonal of $\Sigma B$.
\end{proof}

The previous proposition and theorem \ref{struct triangulada} give rise to the following result.

\begin{thm}\label{StablTriang} If $\mc{D}$ is a stable simplicial descent category then $Ho\mc{D}$ is a triangulated category. In addition, a descent functor $F:\mc{D}\rightarrow \mc{D}'$ induces a functor of triangulated categories $F:Ho\mc{D}\rightarrow Ho\mc{D}'$.
\end{thm}

\begin{cor}\label{cORStablTriang} Let $\mc{D}$ be a stable simplicial descent category such that $\Sigma:\mc{D}\rightarrow \mc{D}$ is an equivalence of categories. Then, if $I$ is a small category, the category of diagrams $(I,\mc{D})$ is a stable simplicial descent category. In particular, $Ho(I,\mc{D})$ is a Verdier's triangulated category, and any functor $f:I\rightarrow J$ induces a triangulated functor $f^\ast: Ho(J,\mc{D})\rightarrow Ho(I,\mc{D})$.
\end{cor}

\begin{proof} If $\mc{D}$ is a simplicial descent category, then $(I,\mc{D})$ inherits a simplicial descent category defined object-wise. In addition, a functor $f$ as above induces a descent functor $f^\ast$. By assumption, $\Sigma^{-1}:\mc{D}\rightarrow\mc{D}$ provides ${\Sigma}^{-1}:(I,\mc{D})\rightarrow (I,\mc{D})$ preserving object-wise equivalences, for all $I$. Therefore, the induced functor $Ho(I,\mc{D})\rightarrow Ho(I,\mc{D})$ is an inverse of $\Sigma:Ho(I,\mc{D})\rightarrow Ho(I,\mc{D})$. Hence $(I,\mc{D})$ is stable of all $I$.
\end{proof}

The above corollary remains valid for a simplicial descent category $\mc{D}$ such that $\Sigma_{I}:Ho(I,\mc{D})\rightarrow Ho(I,\mc{D})$ is an equivalence of categories for all $I$. This holds, for instance, in the following case. Assume the quasi-inverse of $\Sigma$, $\Sigma^{-1}:Ho\mc{D}\rightarrow Ho\mc{D}$ comes from a functor $\underline{\Sigma}:\mc{D}\rightarrow\mc{D}$. Note that $\underline{\Sigma}$ may not be a quasi-inverse of $\Sigma:\mc{D}\rightarrow\mc{D}$. Moreover, assume that the isomorphisms $\alpha:\Sigma\Sigma^{-1}\simeq Id_{Ho\mc{D}}$ and $\beta:\Sigma\Sigma^{-1}\simeq Id_{Ho\mc{D}}$ are isomorphisms of $Fun(\mc{D},\mc{D})[\mrm{E}^{-1}]$. In this case, $\Sigma_I$ is an equivalence of categories as well.

\begin{ejs}\mbox{}\\
\indent \textbf{DG-modules over a DG-category} The cosimplicial descent structure on the category of DG-modules over a fixed DG-category is based on the one of cochain complexes. We then recover the well-known triangulated structure of the derived category of DG-modules over a DG-category \cite{K}.

\textbf{Mixed Hodge complexes} Let $\mc{H}dg$ be the category of mixed Hodge complexes defined in \cite[4.8 and 4.11]{R1}. Consider mixed Hodge complexes
$K=((K_\mathbb{Q},\mrm{W}),(K_\mathbb{C},\mrm{W},\mrm{F}),\alpha)$, $S=((S_\mathbb{Q},\mrm{U}),(S_\mathbb{C},\mrm{U},\mrm{G}),\beta)$ and $T=((T_\mathbb{Q},\mrm{V}),(T_\mathbb{C},\mrm{V},\mrm{H}),\gamma)$. Given morphisms $T\stackrel{g}{\rightarrow}K\stackrel{f}{\leftarrow}S$, then $P=path(f,g)$ is the mixed Hodge complex $((P_\mathbb{Q},\mrm{N}),(P_\mathbb{C},\mrm{N},\mrm{M}),\delta)$ where
$$P^n_\ast=S^n_\ast\oplus K^{n-1}_\ast\oplus T^n_\ast \ \ \ \ \mrm{N}_kP^n_\ast=\mrm{U}_k S^n_\ast\oplus \mrm{W}_{k+1}K_\ast^{n-1}\oplus \mrm{V}_kT^n_\ast\mbox{ for }\ast=\mathbb{Q},\mathbb{C}$$
$$\mrm{M}^kP^n_\mathbb{C}=\mrm{G}^k S^n_\mathbb{C}\oplus \mrm{F}^{k}K_\mathbb{C}^{n-1}\oplus \mrm{H}^kT^n_\mathbb{C}$$
and $\delta$ is the direct sum of $\beta$, $\alpha$ and $\gamma$.\\
The cofiber sequences defined through the $path$ functor induce a left triangulated structure on the derived category $\mc{H}dg[\mrm{E}^{-1}]$, $\mrm{E}$=quasi-isomorphisms. As in the case of filtered complexes, it becomes a Verdier's triangulated category if we consider bounded-below cochain complexes. This triangulated structure is related to the one given in \cite{Be} and \cite{H} (recall that a mixed Hodge complex in our sense becomes a mixed Hodge complex in the sense of loc. cit. after applying decalage functor $Dec$ to the weight filtration $\mrm{W}$).

\textbf{Fibrant spectra} The category $Sp$ of fibrant spectra has a structure of cosimplicial descent category where the simple functor is the homotopy limit. In the proof of \cite[proposition 5.14]{R1} it is checked that the resulting fiber sequences are the usual `homotopy fiber sequences' coming from the Quillen model structure on $Sp$. Therefore, we deduce the classical triangulated structure of the stable homotopy category of fibrant spectra.
\end{ejs}

\end{document}